\documentclass{amsart}
\usepackage{amssymb, amsmath}
\usepackage{wrapfig,tikz}
\usetikzlibrary{calc,through,backgrounds}
\usepackage{amsmath}
\usepackage{amsfonts}
\usepackage{amssymb}
\title{Hall bases for free Leibniz algebras }
\author{M. Shahryari}
\address{M. Shahryari: Department of Pure Mathematics,  Faculty of Mathematical
Sciences, University of Tabriz, Tabriz, Iran}
\email{mshahryari@tabrizu.ac.ir}

\newcommand{\lp}{\dashv}
\newcommand{\rp}{\vdash}
\newcommand{\M}{\mathrm{M}^{\pm}(X)}
\newcommand{\D}{\mathrm{D}^{\pm}(X)}
\newcommand{\A}{\mathrm{A}^{\pm}(X)}
\newcommand{\Ad}{\mathrm{A}^{\pm}_d(X)}
\newcommand{\Leib}{\mathrm{Leib}(X)}
\newcommand{\pr}{\prime}
\newcommand{\prr}{\prime\prime}

\newtheorem{proposition}{Proposition}

\newtheorem {theorem}{Theorem}

\newtheorem{definition}{Definition}
\begin{document}

\maketitle
\begin{abstract}
The aim of this article is to introduce Hall bases of free Leibniz algebras. We modify the classical notion of Hall bases for free Lie algebras in order to provide the similar construction for the case of Leibniz algebras.
\end{abstract}


%1111111111111111111111111111111111111111111111111111111111111111111111111
%%%%%%%%%%%%%%%%%%%%%%%%%%%%%%%%%%%%%%%%%%%%%%%%%%%%%%%%%%%%%%%%%%%%%%
\vspace{3cm}

A Leibniz algebra is a vector space $L$ over a field $\mathbb{K}$ with some bilinear product $[- , -]$ which satisfies the Leibniz identity
$$
[[x,y],z]=[[x,z],y]+[x,[y,z]].
$$
One way of obtaining such algebras is to use a di-algebra $D$. This is a vector space equipped with two bilinear associative products $\lp$ and $\rp$,  and the laws
\begin{eqnarray*}
(x\lp y)\lp z&=& x\lp(y\rp z)\\
(x\rp y)\lp z&=& x\rp(y\lp z)\\
(x\lp y)\rp z&=& x\rp(y\rp z)
\end{eqnarray*}
If we define $[x, y]=x\lp y-y\rp x$, then $D$ becomes a Leibniz algebra. Suppose $\Leib$ denotes the free Leibniz algebra over a set $X$. Some combinatorial problems require a linear basis of this algebra. Such  linear bases already constructed in the case of free Lie algebras using the notion of Hall trees. In this article, we modify the classical approach in such a way that a Hall basis for $\Leib$ is obtained. This modification is not trivial, but after giving the suitable definitions of the necessary notions, the rest of the work is in some sense similar to the classical case of Lie algebras (see \cite{Reu} for the classical constructions).

The reader who likes to review  the fundamental notions of Leibniz  and di-algebras   may use the references \cite{Bar} and \cite{Loday}. For a history of several linear bases for associative and non-associative algebras, one may consult \cite{Bokut}.

\section{Signed Hall trees}
Let $X$ be a non-empty set. A signed  tree is defined inductively as follows:\\

1- Every element of $X$ is a signed tree.

2- If $t_1$ and $t_2$ are signed  trees, then $(t_1, t_2)_-$ and $(t_1, t_2)_+$ are also signed  trees.\\

So, every signed  tree is in fact an element of $X$ or a triple consisting two smaller signed  trees and a $\pm$ sign. Elements of $X$ have length one and if $t=(t_1, t_2)_{\pm}$, then $|t|=|t_1|+|t_2|$. We denote $t_1$, the immediate left part of $t$, by $t^{\pr}$ and $t_2$, the immediate right part of $t$, by $t^{\prr}$. Let $\M$ be the set of all signed  trees. A Hall order is a linear ordering $\leq$ on $\M$ such that $t\leq t^{\prr}$, for all $t$. It is easy to see that there are many Hall orders on $\M$. For example, suppose $X=\{ x, y\}$, and consider the following ordering
\begin{align*}
&x > y > (x,x)_+ > (x,x)_- > (x,y)_+ > (x,y)_- > (y,x)_+ > (y,x)_- > (y,y)_+ > \\
&(y,y)_- > (x,(x,x)_+)_+ > (x,(x,x)_+)_- > (x,(x,x)_-)_+ > (x,(x,x)_-)_- >\\
&(x, (x,y)_+)_+ > (x,(x,y)_+)_- > (x,(x,y)_-)_+ > (x,(x,y)_-)_- >\cdots .
\end{align*}
From now on, we assume that $\leq$ is a fixed Hall order given over $\M$. A subset $H\subseteq \M$ is called a Hall set, if it satisfies the following requirements\\

1- Every element of $X$ belongs to $H$.

2- If $t\in \M\setminus X$, then $t\in H$ if and only if, $t^{\pr}, t^{\prr}\in H$, $t^{\pr}\leq t^{\prr}$, and $t^{\pr}\in X$ or $t^{\prr}\leq (t^{\pr})^{\prr}$.\\

It is easy to see that for any fixed Hall order, there is a unique Hall set in $\M$. Again, for example, if we consider the above ordering, then the following set is a Hall set
\begin{eqnarray*}
H&=&\{ x, y, (x,x)_+, (x,x)_-, (y,x)_+, (y,x)_-, (y,y)_+, (y,y)_-, ((x,x)_+,x)_+,\\
&\ &((x,x)_+,x)_-, ((x,x)_-,x)_+, ((x,x)_-)_-, ((x,x)_+,y)_+, \ldots\}.
\end{eqnarray*}
Every element of $H$ will be called a signed Hall tree.  A standard sequence is a tuple $s=(t_1, t_2, \ldots, t_n; j)$, where every $t_i$ is a signed Hall tree, $1\leq j\leq n$ is an integer (which is called the middle of the sequence), and for any $i$, the signed Hall tree $t_i$ belongs to $X$ or otherwise
$$
t_n, \ldots, t_{i+1}\leq t_i^{\prr}.
$$
There are many examples of standard sequences, for example, if every $t_i$ is an element of $X$, or if $t_1\geq t_2\geq \cdots\geq t_n$, then obviously $s$ is a standard sequence. A rise in $s$ is an index $i$ such that $t_i\leq t_{i+1}$. If further we have $t_{i+1}\geq t_{i+2}, \ldots, t_n$, then we say that $i$ is a legal rise.

\begin{definition}
Let $s=(t_1, t_2, \ldots, t_n; j)$ be a standard sequence with a legal rise $i$. The  rewriting of $s$ in the place $i$ is defined as follows.\\

1- If $i\leq j$, then it is the sequence
$$
s^{\pr}=(t_1, \ldots, t_{i-1}, (t_i, t_{i+1})_+, \ldots, t_n; j-1).
$$

2- If $j\leq i$, then it is the sequence
$$
s^{\pr}=(t_1, \ldots, t_{i-1}, (t_i, t_{i+1})_-, \ldots, t_n; j).
$$
\end{definition}
It is easy to check that $s^{\pr}$ is again a standard sequence. Let $s_1$ and $s_2$ be two standard sequences. The notation $s_1\to s_2$ indicates that $s_2$ can be obtained from $s_1$ by a finite number of rewriting operations. The proof of the following proposition is completely similar to the case of classical standard sequences of Hall trees and so we omit the proof. The reader can consult \cite{Reu}, page 86.

\begin{proposition}
Suppose  $s$, $s_1$, and $s_2$ are standard sequences. \\

1- If $s\to s_1$ and $s\to s_2$, then there exists a standard sequence $r$ such that $s_1\to r$ and $s_2\to r$.

2- There is a standard sequence $r$ consisting of elements of $X$, such that $r\to s$.

3- There exists a standard decreasing sequence $r$, such that $s\to r$.
\end{proposition}

\section{Free Leibniz algebra and Leibniz polynomials}

A di-semigroup is a non-empty set $M$ with two  associative binary operations $\lp$ and $\rp$ satisfying the identities
\begin{eqnarray*}
(x\lp y)\lp z&=& x\lp(y\rp z)\\
(x\rp y)\lp z&=& x\rp(y\lp z)\\
(x\lp y)\rp z&=& x\rp(y\rp z)
\end{eqnarray*}
If $x_1, \ldots, x_n\in M$ are arbitrary elements, then applying any sequence of the operations $\lp$ and $\rp$, we obtain a di-semigroup word on these elements, for example
$$
y=((x_1\lp x_2)\rp (x_3\lp x_4))\lp (x_5\rp x_6)
$$
is a di-semigroup word on elements $x_1, \ldots, x_6$. Any such word can be represented by a rooted planar tree whose nodes are indexed by one of the symbols $\lp$ or $\rp$. In the case of the above word the corresponding tree is the following:\\

\begin{center}
\begin{tikzpicture}
[level distance=10mm,
every node/.style={fill=red!20,circle,inner sep=.5pt},
level 1/.style={sibling distance=20mm},%nodes={fill=red!45}},
level 2/.style={sibling distance=10mm},%nodes={fill=red!30}},
level 3/.style={sibling distance=5mm}]%,nodes={fill=red!25}}]
\node {$\dashv$}[grow'=up]
child[solid,level distance=10mm] {node {$\vdash$}
child[solid] {node {$\dashv$}
child {node {$x_1$}}
child {node {$x_2$}}
}
child {node {$\dashv$}
child[solid] {node {$x_3$}}
child {node {$x_4$}}
}
}
child {node {$\vdash$}
child {node {$ x_5 $}
}
child {node {$x_6$}}
};
\end{tikzpicture}
\end{center}

If we move from the root toward leafs and in any node follow the directions indicated by the symbols $\lp$ and $\rp$, then we arrive a leaf $x_j$ which is called the middle of $y$. In our example the middle is $x_3$:\\

\begin{center}
\begin{tikzpicture}
[level distance=10mm,
every node/.style={fill=red!20,circle,inner sep=.5pt},
level 1/.style={sibling distance=20mm},%nodes={fill=red!45}},
level 2/.style={sibling distance=10mm},%nodes={fill=red!30}},
level 3/.style={sibling distance=5mm}]%,nodes={fill=red!25}}]
\node {$\dashv$}[grow'=up]
child[dashed,level distance=10mm] {node {$\vdash$}
child[solid] {node {$\dashv$}
child {node {$x_1$}}
child {node {$x_2$}}
}
child {node {$\dashv$}
child {node {$x_3$}}
child [solid]{node {$x_4$}}
}
}
child {node {$\vdash$}
child {node {$ x_5 $}
}
child {node {$x_6$}}
};
\end{tikzpicture}
\end{center}

It is not hard to see that the laws of di-semigroup implies
$$
y=x_1\rp x_2\rp x_3\lp x_4\lp x_5\lp x_6,
$$
and this expression does not depend on paranthesing. So, any di-semigroup word can be represented as a normal form
$$
x_1\rp \cdots\rp x_j\lp \cdots\lp x_n.
$$
In the case of free di-semigroups, this normal form is also unique.

The free di-semigroup over a set $X$ can be constructed as follows: On the set $\M$ define two operations
$$
t_1\lp t_2=(t_1, t_2)_-, \ \ t_1\rp t_2=(t_1, t_2)_+.
$$
Let $R$ be the ideal generated by all laws defining a di-semigroup. Then
$$
\D=\frac{\M}{R}
$$
is the free di-semigroup on $X$. Every element of this free di-semigroup has a unique representation of the normal form
$$
x_1\rp \cdots\rp x_j\lp \cdots\lp x_n.
$$
We also call elements of $\D$ monomials. Let $\mathbb{K}$ be a field and $\A$ be the vector space with basis $\D$ over $\mathbb{K}$. If we extend the operation $\lp$ and $\rp$ bilinearly, $\A$ becomes the free di-algebra over $X$. This free di-algebra has also the structure of Leibniz algebra, since we can define the Leibniz bracket
$$
[P,Q]=P\lp Q-Q\rp P,
$$
For all $P, Q\in \A$. The free Leibniz algebra over $X$ is the smallest Leibniz subalgebra of $\A$ which includes $X$. We denote it by $\Leib$. The elements of $\Leib$ will be called Leibniz polynomials. For any signed tree $t$, we define a monomial $(t)\in \D$ by induction:\\

1- For any $x\in X$ we have $(x)=x$.

2- If $t=(t_1,t_2)_+$, then $(t)=(t_1)\rp (t_2)$.

3- If $t=(t_1,t_2)_-$, then $(t)=(t_1)\lp (t_2)$.

\begin{proposition}
Every element of $\D$ can be uniquely represented as
$$
(t_1)\rp\cdots\rp (t_j)\lp\cdots \lp (t_n),
$$
for some decreasing sequence of signed Hall trees $t_1\geq t_2\geq \cdots\geq t_n$ and some integer $1\leq j\leq n$.
\end{proposition}

\begin{proof}
The general idea of this proof is the same as \cite{Reu}, but in details it has some differences. For any standard sequence $s=(t_1, \ldots, t_n;j)$, define
$$
(s)=(t_1)\rp\cdots\rp(t_j)\lp\cdots\lp (t_n).
$$
We show that if $s\to s^{\pr}$, then $(s)=(s^{\pr})$. There are two cases: if
$$
s^{\pr}=(t_1, \ldots, (t_i, t_{i+1})_+, \ldots, t_n;j-1),
$$
then we have
\begin{eqnarray*}
(s^{\pr})&=&(t_1)\rp\cdots\rp((t_i)\rp(t_{i+1}))\rp\cdots\rp (t_j)\lp\cdots\lp (t_n)\\
         &=&(s),
\end{eqnarray*}
because of the associativity of $\rp$. If we have
$$
s^{\pr}=(t_1, \ldots, (t_i, t_{i+1})_-, \ldots, t_n;j),
$$
then
\begin{eqnarray*}
(s^{\pr})&=&(t_1)\rp\cdots\rp (t_j)\lp\cdots\lp((t_i)\lp(t_{i+1}))\lp\cdots\lp (t_n)\\
         &=&(s),
\end{eqnarray*}
because of the associativity of $\lp$. Now, suppose $w\in \D$ and $s$ is the standard sequence of letters of $w$ in its normal form. Then $s\to r$, where $r$ is a decreasing sequence of signed Hall trees. So, we have
$$
w=(s)=(r)=(t_1)\rp\cdots\rp (t_j)\lp\cdots \lp (t_n),
$$
For signed Hall trees $t_1\geq t_2\geq \cdots\geq t_n$ and some integer $1\leq j\leq n$. To prove the uniqueness, suppose in the same time we have
$$
w=(u_1)\rp\cdots\rp (u_k)\lp\cdots \lp (u_m),
$$ for a sequence of signed Hall trees $u_1\geq u_2\geq\cdots\geq u_m$. Let $s=(t_1, \ldots, t_n;j)$ and $r=(u_1, \ldots, u_m;k)$. Then clearly, we have $(s)=w=(r)$. By the proposition 1, there are two standard sequences of letters $s^{\pr}$ and $r^{\pr}$, such that $s^{\pr}\to s$ and $r^{\pr}\to r$. So we have $(s^{\pr})=(r^{\pr})$. Since $\D$ is free, we have $s^{\pr}=r^{\pr}$. Again by the proposition 1, there is a standard sequence $p$, such that $s\to p$ and $r\to p$. But $s$ and $r$ are decreasing, so they have no more rewritings. This shows that $s=p=r$.
\end{proof}

Every element of $\D$ of the form $(t)$ will be called a signed Hall word, if $t$ is such a tree. Hence for every signed Hall word $w$, there is exactly one signed Hall tree $t$ such that $w=(t)$. Also, it is now clear that every monomial is a normal product of a decreasing set of signed Hall words, and this representation is unique.\\

Recall that the monomial $(t)$ was defined for any signed Hall tree. Similarly we can define a Leibniz polynomial $[t]$ by induction:\\

1- For any $x\in X$, we define $[x]=x$.

2- If $t=(t_1,t_2)_+$, then $[t]=[[t_1],[t_2]]$.

3- If $t=(t_1,t_2)_-$, then $[t]=[[t_2],[t_1]]$.\\

\begin{theorem}
The set of all expressions of the form
$$
[t_1]\rp\cdots\rp [t_j]\lp\cdots \lp [t_n],
$$
with $n\geq 1$, $t_1\leq t_2\leq \cdots \leq t_n$, and $1\leq j\leq n$, is a basis of $\A$.
\end{theorem}

\begin{proof}
Through the proof, we denote the set of all such polynomials by $B$. Let $s=(t_1, \ldots, t_n;j)$ be a standard sequence. Define
$$
[s]=[t_1]\rp\cdots\rp [t_j]\lp\cdots \lp [t_n].
$$
For any legal rise $i$ in $s$, we define new sequences $\lambda_i(s)$ and $\rho_i(s)$ as follows:\\

1- If $i+1\neq j$, then $\lambda_i(s)=(t_1, \ldots, (t_i, t_{i+1})_+, \ldots, t_n;j-1)$.

2- If $i+1= j$, then $\lambda_i(s)=(t_1, \ldots, (t_i, t_{i+1})_-, \ldots, t_n;j-1)$.

3- $\rho_i(s)=(t_1, \ldots, t_{i+1}, t_i, \ldots, t_n;j)$.\\

It is not hard to check that both $\lambda_i(s)$ and $\rho_i(s)$ are standard sequences. A calculation shows that if $i+1\neq j$, then $ [s]=[\lambda_i(s)]+[\rho_i(s)]$, and if $i+1=j$, then $[s]=[\rho_i(s)]-[\lambda_i(s)]$. So,  we argue with induction on the number
$$
k=n+ the\ number\ of\ inversions.
$$
Note that an inversion in a sequence $s$ is a pair of indices $p$ and $q$ such that $p<q$ and $t_p<t_q$. Now, the length of $\lambda_i(s)$ is $n-1$ and the number of inversions of $\rho_i(s)$ is smaller than those of $s$, therefore $[s]\in \langle B\rangle_{\mathbb{K}}$. Now, consider a monomial
$$
w=x_1\rp\cdots\rp x_j\lp\cdots\lp x_n.
$$
We have
$$
w=[(x_1, \ldots, x_n;j)]\in \langle B\rangle_{\mathbb{K}},
$$
and this shows that $\langle B\rangle_{\mathbb{K}}=\A$. We now show that the set $B$ is linearly independent. Without loos of generality, we can assume that $X$ is finite. Suppose $\Ad$ is the homogenous part of $\A$ consisting of the polynomials of degree $d$. It is clear that $\dim \Ad=|X|^{d+1}$. Since every monomial of degree $d$ can be represented as
$$
(t_1)\rp\cdots\rp (t_j)\lp\cdots \lp (t_n),
$$
for some decreasing sequence of signed Hall trees $t_1\geq t_2\geq \cdots\geq t_n$ and some integer $1\leq j\leq n$, with $\sum |t_i|=d$, so the set
$$
B_0=\{ (t_1)\rp\cdots\rp (t_j)\lp\cdots \lp (t_n): n\geq 1, t_1\geq t_2\geq \cdots\geq t_n, t_i\in H, \sum |t_i|=d\}
$$
is a basis of $\Ad$. Now, we have a correspondence
$$
(t_1)\rp\cdots\rp (t_j)\lp\cdots \lp (t_n)\to [t_1]\rp\cdots\rp [t_j]\lp\cdots \lp [t_n],
$$
between $B_0$ and the set
$$
B_1=\{ [t_1]\rp\cdots\rp [t_j]\lp\cdots \lp [t_n]: n\geq 1, t_1\geq t_2\geq \cdots\geq t_n, t_i\in H, \sum |t_i|=d\}.
$$
As we saw, the later generates $\Ad$, so it is also a basis of $\Ad$. This shows that $B$ is linearly independent.
\end{proof}

\section{Hall basis}
Now, consider that $[H]=\{ [t]: t\in H\}$. We prove the main theorem of this article:

\begin{theorem}
The set $[H]$ is a linear basis of $\Leib$.
\end{theorem}

\begin{proof}
By Theorem 1, this set is linearly independent. We know that $\Leib$ is the smallest Leibniz subalgebra of $\A$ which contains $X$. We also have
$$
X\subseteq \langle [H]\rangle_{\mathbb{K}}\subseteq \Leib.
$$
So, we prove that $\langle [H]\rangle_{\mathbb{K}}$ is a Leibniz subalgebra. Equivalently, we show that
$$
t_1, t_2\in H\Rightarrow [[t_1],[t_2]]\in \langle [H]\rangle_{\mathbb{K}}.
$$
Again, without loos of generality, we assume that $X$ is finite. Let $\alpha=(|t_1|+|t_2|, \max(t_1, t_2))$. We prove by induction on $\alpha$, that
$$
[[t_1],[t_2]]=\sum \lambda_i [u_i],
$$
for some scalars $\lambda_i$, and signed Hall trees $u_i$, with $u_i^{\prr}< \max(t_1, t_2)$. Note that the ordering of the set of all such $\alpha$'s is lexicographic. If $\alpha=(2, x)$, then $t_1=y\in X$ and $t_2=y\in X$ and $y<x$. Hence $(y,x)_-\in H$ and so
$$
[[t_1],[t_2]]=[x,y]=[(y,x)_-]\in \langle [H]\rangle_{\mathbb{K}}.
$$
Now, suppose for all $x\in X$, we have $\alpha> (2,x)$. There are two cases:\\

Case 1- Assume that $t_1\leq t_2$. We have three subcases:\\

1-1. Let $t_1\in X$. Then $(t_1, t_2)_+\in H$ and hence
$$
[[t_1],[t_2]]=[(t_1,t_2)_+]\in \langle [H]\rangle_{\mathbb{K}}.
$$

1.2. Let $t_1=(t_1^{\pr}, t_2^{\prr})$, with $t_1^{\prr}\geq t_2$. Then $(t_1, t_2)_+\in H$ and hence
$$
[[t_1],[t_2]]=[(t_1, t_2)_+]\in \langle [H]\rangle_{\mathbb{K}}.
$$

1-3. Let $t_1=(t_1^{\pr},t_2^{\prr})$, with $t_1^{\prr}\leq t_2$. Since $t_1\leq t_1^{\prr}$ and $t_1^{\pr}\leq t_1^{\prr}$, so
$$
t_1\leq t_1^{\prr}\leq t_2, \ \ t_1^{\pr}\leq t_1^{\prr}\leq t_2.
$$
By the Leinbiz identity, we have
\begin{eqnarray*}
[(t_1,t_2)_+]&=& [[t_1],[t_2]]\\
             &=& [[[t_1^{\pr}], [t_1^{\prr}]], [t_2]]\\
             &=& [[[t_1^{\pr}],[t_2]],[t_1^{\prr}]]+[[t_1^{\pr}],[[t_1^{\prr}],[t_2]]].
\end{eqnarray*}
We also have
\begin{eqnarray*}
(|t_1^{\pr}|+|t_2|, \max(t_1^{\pr}, t_2))&=&(|t_1^{\pr}|+|t_2|, t_2)\\
                                         &<&(|t_1|+|t_2|, \max(t_1, t_2)),
\end{eqnarray*}
as well as
\begin{eqnarray*}
(|t_1^{\prr}|+|t_2|, \max(t_1^{\prr}, t_2))&=&(|t_1^{\prr}|+|t_2|, t_2)\\
                                         &<&(|t_1|+|t_2|, \max(t_1, t_2)).
\end{eqnarray*}
Hence, by the induction hypothesis, we have
$$
[[t_1^{\pr}],[t_2]]=\sum \lambda_i [u_i],
$$
for some scalars $\lambda_i$, and signed Hall trees $u_i$, with $u_i^{\prr}\leq \max(t_1^{\pr}, t_2)=t_2$. Similarly, we have
$$
[[t_1^{\prr}],[t_2]]=\sum \mu_j [v_j],
$$
for some scalars $\mu_j$, and signed Hall trees $v_j$, with $v_j^{\prr}\leq \max(t_1^{\prr}, t_2)=t_2$. Note that, we also have $|u_i|=|t_1^{\pr}|+|t_2|$, and $|v_j|=|t_1^{\prr}|+|t_2|$. Therefore
$$
[[t_1],[t_2]]=\sum \lambda_i[[u_i],[t_1^{\prr}]]+\sum\mu_j[[t_1^{\pr}], [v_j]],
$$
and the assertion can be now obtained from the induction hypothesis. \\

Case 2- Now assume that $t_1<t_2$. We have
$$
[[t_1],[t_2]]=[(t_2,t_1)_-],
$$
and the assertion follows from the previous case.
\end{proof}

\end{document}